\documentclass[12pt,a4paper]{article}
\usepackage{amsmath,amssymb,amsfonts}
\usepackage{times}
\usepackage{amsthm}
\usepackage{mathtools}
\usepackage{marvosym}
\usepackage{authblk}
\usepackage{enumerate}
\usepackage[english]{babel}
\usepackage{enumitem}
\usepackage{todonotes}
\usepackage{indentfirst}
\usepackage{geometry}
\usepackage{graphicx}
\usepackage{tabularx}
\usepackage{tocbibind}
\usepackage{tikz}
\usepackage{array}
\usepackage{caption}
\usepackage{setspace}
\usepackage{hyperref}
\usepackage{lipsum}
\usepackage{listings}
\usepackage{multicol,multirow}
\usepackage{xcolor,xspace}
\usepackage[english]{babel}
\usepackage[export]{adjustbox}
\newcommand{\keywords}[1]{\textbf{Keywords:} #1}
\newtheorem{theorem}{Theorem}[section]
\newtheorem{corollary}{Corollary}[theorem]
\newtheorem{lemma}[theorem]{Lemma}

\theoremstyle{definition}

\theoremstyle{remark}

\title{\textbf{A graph product and its Application}}
\author[1]{Bishal Sonar\thanks{Email: bsonarnits@gmail.com}}
\author[2]{Ravi Srivastava\thanks{Corresponding author Email: ravi@nitsikkim.ac.in}}
\affil[1,2]{Department of Mathematics, National Institute of Technology Sikkim, Sikkim 737139, India}
\date{}
\begin{document}

\parskip1ex
\parindent0pt
\maketitle

\begin{abstract}
    \noindent The spectrum of Laplacian and signless Laplacian matrix for a graph product is obtained, where both underlying graphs are regular. As an application of this, we have been able to generate the Kirchhoff Index and Wiener Index and determine the number of spanning trees. Additionally, we derived the conditions necessary for obtaining a Laplacian and signless Laplacian integral product graph.
\end{abstract}


\section{Introduction}
    In graph theory, a graph $\mathcal{G}$ is represented as an ordered pair $(\mathcal{V}, \mathcal{E})$, where $\mathcal{V} = \{v_i : i = 1, 2, \ldots, n\}$ is the set of vertices, and $\mathcal{E} = \{e_i : i = 1, 2, \ldots, m\}$ is the set of edges. The degree of a vertex $v_i$, denoted as $d(v_i)$, is defined as the number of edges incident to $v_i$. The adjacency matrix $A_{\mathcal{G}} = (a_{ij})_{n\times n}$ is a key representation of the graph $\mathcal{G}$, where each entry $a_{ij} = 1$ if there is an edge connecting vertex $v_i$ and vertex $v_j$, and $a_{ij} = 0$ otherwise. Further, two important matrices associated with a graph $\mathcal{G}$ are introduced: the Laplacian matrix $L_{\mathcal{G}} = D_{\mathcal{G}} - A_{\mathcal{G}}$, and the signless Laplacian matrix $Q_{\mathcal{G}} = D_{\mathcal{G}} + A_{\mathcal{G}}$. Here, $D_{\mathcal{G}}$ is a diagonal matrix where the $i^{th}$ diagonal entry is $d(v_i)$, the degree of the vertex $v_i$. The shortest path length between any two vertices $v_i,v_j$ of the graph $\mathcal{G}$ is denoted by $d(v_i,v_j)$. And $D(\mathcal{G})=[d(v_i,v_j)]_{n\times n}$ denotes the distance matrix of $\mathcal{G}$. A graph $\mathcal{G}$ is classified as Laplacian Integral or Signless Laplacian Integral if all the eigenvalues of its Laplacian or signless Laplacian matrix, respectively, are integers. This discussion focuses exclusively on simple, undirected, and finite graphs.

    Various graph products, such as cartesian, tensor, strong, and many more, have been studied. In 1970, R Frucht and F Harary~\cite{frucht1970corona} first defined the corona product for graphs. Later, Cam McLeman and Erin McNicholas~\cite{mcleman2011spectra} introduced the coronal of graphs to obtain the adjacency spectra for arbitrary graphs. Then Shu and Gui~\cite{cui2012spectrum} generalized it and defined the coronal of the Laplacian and the signless Laplacian matrix of graphs. In 2023, SP Joseph~\cite{joseph2023graph} introduced a new product and worked on its adjacency spectra and its application in generating non-cospectral equienergetic graphs. However, work still needs to be done for Laplacian and signless Laplacian spectra. We calculated the Laplacian and signless Laplacian spectrum, generated the Kirchhoff Index, and determined the number of spanning trees. Additionally, we derived the conditions necessary for obtaining a Laplacian and signless Laplacian integral graph of $\mathcal{G}_1\circledast\mathcal{G}_2$.

    In the upcoming sections, you will find preliminary results, followed by the definition in the third section. The fourth section delves into the spectrum of the Laplacian and signless Laplacian, while the fifth section explores various applications of the graph product.

    \section{Preliminaries}
    \begin{lemma}\label{L2}~\cite{bapat2010graphs}
        Let $\mathcal{S}_1,\mathcal{S}_2,\mathcal{S}_3,$ and $\mathcal{S}_4$ be matrix of order $q_1\times q_1,~ q_1\times q_2, ~q_2\times q_1, ~q_2\times q_2$ respectively with $\mathcal{S}_1$ and $\mathcal{S}_4$ are invertible. Then 
            \begin{align*}
                \det\begin{bmatrix}
                \mathcal{S}_1&\mathcal{S}_2\\\mathcal{S}_3&\mathcal{S}_4
        \end{bmatrix}&=\det(\mathcal{S}_1)\det(\mathcal{S}_4-\mathcal{S}_3\mathcal{S}_1^{-1}\mathcal{S}_2)\\
        &=\det(\mathcal{S}_4)\det(\mathcal{S}_1-\mathcal{S}_2\mathcal{S}_4^{-1}\mathcal{S}_3)
        \end{align*}
    \end{lemma}
        \subsection{Kronecker product}
        Consider two matrices $\mathcal{P}=(p_{ij})$ of size $p_1\times q_1$ and $\mathcal{Q}=(q_{ij})$ of size $p_2\times q_2$. The Kronecker product, denoted as $\mathcal{P}\otimes \mathcal{Q}$, is a powerful mathematical operation that results in a new matrix of dimensions $p_1p_2 \times q_1q_2$. This operation is performed by multiplying each element $p_{ij}$ of matrix $\mathcal{P}$ by the entire matrix $\mathcal{Q}$, effectively scaling $\mathcal{Q}$ by $p_{ij}$. As outlined by Neumaier (1992)~\cite{neumaier1992horn}, this concept provides a structured way to construct a larger matrix from two smaller matrices, finding applications in various fields such as signal processing and theoretical physics.\\ \\
        \textbf{Properties:} \vspace{- 0.3 cm}
        \begin{itemize}
            \item It is associative.
            \item $(\mathcal{P}\otimes \mathcal{Q})^T=\mathcal{P}^T\otimes \mathcal{Q}^T$.
            \item $(\mathcal{P}\otimes \mathcal{Q})(\mathcal{R} \otimes \mathcal{S})=\mathcal{PR}\otimes \mathcal{QS}$, given the product $P\mathcal{R}$ and $\mathcal{Q}\mathcal{S}$ exists. 
            \item $(\mathcal{P}\otimes \mathcal{Q})^{-1}=\mathcal{P}^{-1}\otimes \mathcal{Q}^{-1}$, given $\mathcal{P}$ and $\mathcal{Q}$ are non-singular matrices.
            \item If $\mathcal{P}$ and $\mathcal{Q}$ are $p\times p$ and $q\times q$ matrices, then $\det(\mathcal{P}\times \mathcal{Q})=(\det \mathcal{P})^q(\det \mathcal{Q})^p.$
        \end{itemize}

        \subsection{Definition}~\cite{mcleman2011spectra}\label{Def2}
            Let $\mathcal{G}$ be a graph with $n$ vertices. The coronal of $\mathcal{G}$, denoted by $\chi_{\mathcal{G}}(x)$, is defined as the sum of all the entries in the matrix $(xI_n - A_{\mathcal{G}})^{-1}$, where $A_{\mathcal{G}}$ is the adjacency matrix of the graph $\mathcal{G}$, and $I_n$ is the identity matrix of order $n$.
            \begin{equation}
                \chi_{\mathcal{G}}(x)=\textbf{1}_n^T(xI_n-A_{\mathcal{G}})^{-1}\textbf{1}_n,
            \end{equation}
            where $\textbf{1}_n$ is all one column matrix.\\
            The coronal of Laplacian and signless Laplacian matrix was introduced by Shu and Gui \cite{cui2012spectrum} and is given by $\chi_{L_{\mathcal{G}}}=1_n(xI_n-L_{\mathcal{G}})^{-1}1_n$ and $\chi_{Q_{\mathcal{G}}}=1_n(xI_n-Q_{\mathcal{G}})^{-1}1_n$ respectively.
        
\section{Spectral properties}
        
    \subsection{Definition}~\cite{joseph2023graph}~\label{Def1}
        Let $\mathcal{G}_1=(\mathcal{V}_1,\mathcal{E}_1)$ be a graph of size $e_1$ and order $n_1$ with a set of vertex $\{u_i:i=1,2,\cdots,n_1\}$ and $\mathcal{G}_2=(\mathcal{V}_2,\mathcal{E}_2)$ be a graph of size $e_2$ and order $n_2$ with vertex set $\{v_i:1,2,\cdots,n_2\}$. Then, the graph product is defined as follows:
        \begin{enumerate}
            \item The vertex set of $\mathcal{G}_1\circledast\mathcal{G}_2$ is given by\\
           $\{a_{11},a_{12},\cdots,a_{1n_2},a_{21},a_{22},\cdots,a_{2n_2},\cdots,a_{n_11},a_{n_12},\cdots,a_{n_1n_2},b_{11},b_{12},\cdot,b_{1n_2},b_{21},b_{22},\cdots, \\ b_{2n_2},\cdots,b_{n_11},b_{n_12},\cdots,b_{n_1n_2}\}$.
           \item The set of edges of $\mathcal{G}_1\circledast\mathcal{G}_2$ consists of the following three types of edges:\\
           \begin{itemize}\vspace{-.5cm}
               \item If the edge $(u_i,u_j)\in E(\mathcal{G}_1)$, then the edges $(a_{ik},a_{jl});$ for $1\leq k,l\leq n_2$ belong to $\mathcal{G}_1\circledast\mathcal{G}_2$.
               \item If the edge $(v_i,v_j)\in E(\mathcal{G}_2)$, then the edges $(b_{ri},b_{rj});$ for $1\leq r\leq n_1$ belongs to $\mathcal{G}_1\circledast\mathcal{G}_2$.
               \item For $1\leq i\leq n_1$, the edges $(a_{ip},b_{iq});$ for $1\leq p,q \leq n_2$ belongs to $\mathcal{G}_1\circledast\mathcal{G}_2$.
           \end{itemize}
        \end{enumerate}

    \begin{figure}
            \centering
            \includegraphics[scale=0.75]{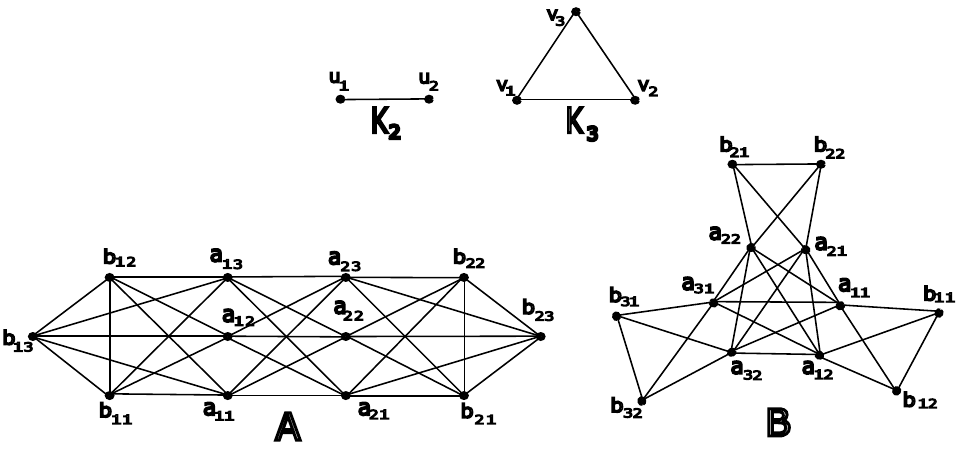}
            \caption{$A=K_2\circledast K_3$ and $B=K_3\circledast K_2.$}
            \label{F1}
    \end{figure}
        
    The product graph represents the total vertices count by $2n_1n_2$. When it comes to the edges, there are different types, each with its own count: $n_2^2e_1$ edges of the first type, $n_1e_2$ edges of the second type, and $n_1n_2^2$ edges of the third type. Therefore, to sum up, the overall number of edges in the product graph can be calculated using the formula $n_2^2(n_1+e_1)+n_1e_2$.\\

\subsection{Laplacian Spectrum for Regular graphs}

        \begin{theorem}~\label{Th3}
            Let $\mathcal{G}_1$ and $\mathcal{G}_2$ be two regular graphs of order $n_1$ and $n_2$ with regularity $r_1$ and $r_2$ respectively. Also let the Laplacian eigenvalues be $\mu_{i1},\mu_{i2},\ldots,\mu_{in_i}$, $i=1,2.$ Then the Laplacian characteristic polynomial of $\mathcal{G}_1\circledast \mathcal{G}_2$ is\\
            $f(L_{\mathcal{G}_1\circledast \mathcal{G}_2},x)=(x-r_1n_2-n_2)^{n_1(n_2-1)}f(L_{\mathcal{G}_2},(x-n_2))^{n_1}\cdot \prod_{i=1}^{n_1}\big[x-n_2-n_2\mu_{1i}-n_2\chi_{L_{\mathcal{G}_2}}(x-n_2).$
        \end{theorem}
        \begin{proof}
            Following the definition~\ref{Def1} of $\mathcal{G}_1\circledast \mathcal{G}_2$, we have
            $$A_{\mathcal{G}_1\circledast \mathcal{G}_2}=\begin{bmatrix}
                A_{\mathcal{G}_1}\otimes J_{n_2} &I_{n_1}\otimes J_{n_2} \\
                I_{n_1}\otimes J_{n_2} &I_{n_1}\otimes A_{\mathcal{G}_2}
            \end{bmatrix}.$$
            $$D_{\mathcal{G}_1\circledast \mathcal{G}_2}= \begin{bmatrix}
                n_2(r_1+1)I_{n_1n_2} & 0_{n_1n_2}\\
                0_{n_1n_2} & (r_2+n_2)I_{n_1n_2}
            \end{bmatrix}.$$
            The Laplacian matrix of $\mathcal{G}_1\circledast \mathcal{G}_2$ is
            $$L_{\mathcal{G}_1\circledast \mathcal{G}_2}=\begin{bmatrix}
                n_2(r_1+1)I_{n_1n_2}-A_{\mathcal{G}_1}\otimes J_{n_2} &-I_{n_1}\otimes J_{n_2} \\
                -I_{n_1}\otimes J_{n_2} &I_{n_1}\otimes (n_2I_{n_2}+L_{\mathcal{G}_2})
            \end{bmatrix}.$$
            Then using Lemma~\ref{L2} the characteristics polynomial of Laplacian matrix of $\mathcal{G}_1\circledast \mathcal{G}_2$ is,\\
            \begin{align}
                &f(L_{\mathcal{G}_1\circledast \mathcal{G}_2},x)=\det\big[xI_{2n_1n_2}-L_{\mathcal{G}_1\circledast \mathcal{G}_2}\big]\notag\\
                &=\det\begin{bmatrix}
                    (x-n_2r_1-n_2)I_{n_1n_2}+A_{\mathcal{G}_1}\otimes J_{n_2} &I_{n_1}\otimes J_{n_2}\notag\\
                I_{n_1}\otimes J_{n_2} &I_{n_1}\otimes(x-n_2)I_{n_2}-L_{\mathcal{G}_2}
                \end{bmatrix}\notag \\ 
                &=\det\big[I_{n_1}\otimes(x-n_2)I_{n_2}-L_{\mathcal{G}_2}\big] \cdot\det\Big[\{(x-n_2r_1-n_2)I_{n_1n_2}+A_{\mathcal{G}_1}\otimes J_{n_2}-\{I_{n_1}\otimes J_{n_2}\}\notag \\ &\{I_{n_1}\otimes(x-n_2)I_{n_2}-L_{\mathcal{G}_2}\}^{-1}\{I_{n_1}\otimes J_{n_2})\}\Big]\notag\\
                &=\det(I_{n_1})^{n_2}\cdot\det\big[(x-n_2)I_{n_2}-L_{\mathcal{G}_2}\big]^{n_1}\cdot \det\Big[\{(x-r_1n_2-n_2)I_{n_1n_2}+A_{\mathcal{G}_1}\otimes J_{n_2}\}+\notag\\
                &\{I_{n_1}\otimes\chi_{L_{\mathcal{G}_2}}(x-n_2)J_{n_2}\}\big]\notag\\
                &=f(L_{\mathcal{G}_2},(x-n_2))^{n_1}\cdot\prod((x-r_1n_2-n_2)+\alpha_i\beta_j), \text{here the product}\text{is over the eigenvalues } \notag\\
                &{\alpha_i}'s \text{ and } {\beta_j}'s \text{ of } A_{\mathcal{G}_1}-\chi_{L_{\mathcal{G}_2}}(x-n_2)I_{n_1} \text{ and } J_{n_2}\text{respectively.}\notag \\
                &=(x-r_1n_2-n_2)^{n_1(n_2-1)}f(L_{\mathcal{G}_2},(x-n_2))^{n_1}\cdot \prod_{i=1}^{n_1}\big[(x-r_1n_2-n_2)+n_2(\lambda_{1i}-\chi_{L_{\mathcal{G}_2}}(x-n_2))\big] \notag\\
                &=(x-r_1n_2-n_2)^{n_1(n_2-1)}f(L_{\mathcal{G}_2},(x-n_2))^{n_1}\cdot \prod_{i=1}^{n_1}\big[x-n_2-n_2\mu_{1i}-n_2\chi_{L_{\mathcal{G}_2}}(x-n_2)\big]\label{Eqn4}
            \end{align}
           
          \noindent So, $f(L_{\mathcal{G}_1\circledast \mathcal{G}_2},x)=(x-r_1n_2-n_2)^{n_1(n_2-1)}f(L_{\mathcal{G}_2},(x-n_2))^{n_1}\cdot \prod_{i=1}^{n_1}\big[x-n_2-n_2\mu_{1i}-n_2\chi_{L_{\mathcal{G}_2}}(x-n_2)$.
          We used different results from the Kronecker product mentioned in the preliminary section to get this result.
        \end{proof}

        \begin{corollary}~\label{C1}
            For $n_2=1$, the Laplacian characteristics polynomial is the same as the Laplacian characteristics polynomial of the corona product of two graphs.
        \end{corollary}
        \begin{proof}
            From equation (\ref{Eqn4}) we have;
            \begin{align*}
                f(L_{\mathcal{G}_1\circledast \mathcal{G}_2},x)&=(x-r_1n_2-n_2)^{n_1(n_2-1)}f(L_{\mathcal{G}_2},(x-n_2))^{n_1}\cdot \prod_{i=1}^{n_1}\big[x-n_2-n_2\mu_{1i}-\\
                &n_2\chi_{L_{\mathcal{G}_2}}(x-n_2)\big]\\
                &=f(L_{\mathcal{G}_2},(x-1))^{n_1}\cdot \prod_{i=1}^{n_1}\big[x-1-\mu_{1i}-\chi_{L_{\mathcal{G}_2}}(x-1)\big].
            \end{align*}
            
            \noindent Hence, $f(L_{\mathcal{G}_1\circledast \mathcal{G}_2},x)=f(L_{\mathcal{G}_2},(x-1))^{n_1}\cdot \prod_{i=1}^{n_1}\big[x-1-\mu_{1i}-\chi_{L_{\mathcal{G}_2}}(x-1)\big]=$ Laplacian characteristics polynomial of Corona product of $\mathcal{G}_1$ and $\mathcal{G}_2.$
        \end{proof}

        \begin{corollary}
         Let $\mathcal{G}_1$ and $\mathcal{G}_2$ be two co-spectral graphs and $\mathcal{G}$ be any arbitrary  graph, then 
            \begin{enumerate}
                \item $\mathcal{G}_1\circledast \mathcal{G}$ and $\mathcal{G}_2\circledast \mathcal{G}$ are $L$-co-spectral.
                \item $\mathcal{G}\circledast \mathcal{G}_1$ and $\mathcal{G}\circledast \mathcal{G}_2$ are $L$-co-spectral if $\chi_{L_{\mathcal{G}_1}}(x)=\chi_{L_{\mathcal{G}_2}}(x).$
            \end{enumerate}
     \end{corollary}

\subsection{Signless Laplacian Spectrum for Regular graphs}
    
    \begin{theorem}~\label{Th4}
            Let $\mathcal{G}_1$ and $\mathcal{G}_2$ be two regular graphs of order $n_1$ and $n_2$ with regularity $r_1$ and $r_2$ respectively. Also let the Laplacian eigenvalues be $\nu_{i1},\nu_{i2},\ldots,\nu_{in_i}$, $i=1,2.$ Then the signless Laplacian characteristic polynomial of $\mathcal{G}_1\circledast \mathcal{G}_2$ is\\
            $f(Q_{\mathcal{G}_1\circledast \mathcal{G}_2},x)=(x-r_1n_2-n_2)^{n_1(n_2-1)}f(Q_{\mathcal{G}_2},(x-n_2))^{n_1}\cdot \prod_{i=1}^{n_1}\big[x-n_2-n_2\nu_{1i}-n_2\chi_{Q_{\mathcal{G}_2}}(x-n_2)\big].$
    \end{theorem}
    \begin{proof}
        Following the definition~\ref{Def1} of $\mathcal{G}_1\circledast \mathcal{G}_2$ we have
            $$A_{\mathcal{G}_1\circledast \mathcal{G}_2}=\begin{bmatrix}
                A_{\mathcal{G}_1}\otimes J_{n_2} &I_{n_1}\otimes J_{n_2} \\
                I_{n_1}\otimes J_{n_2} &I_{n_1}\otimes A_{\mathcal{G}_2}
            \end{bmatrix}.$$
            $$D_{\mathcal{G}_1\circledast \mathcal{G}_2}= \begin{bmatrix}
                n_2(r_1+1)I_{n_1n_2} & 0_{n_1n_2}\\
                0_{n_1n_2} & (r_2+n_2)I_{n_1n_2}
            \end{bmatrix}.$$
            The signless Laplacian matrix of $\mathcal{G}_1\circledast \mathcal{G}_2$ is
            $$Q_{\mathcal{G}_1\circledast \mathcal{G}_2}=\begin{bmatrix}
                n_2(r_1+1)I_{n_1n_2}+A_{\mathcal{G}_1}\otimes J_{n_2} &I_{n_1}\otimes J_{n_2} \\
                I_{n_1}\otimes J_{n_2} &I_{n_1}\otimes n_2I_{n_1}+Q_{\mathcal{G}_2}
            \end{bmatrix}.$$
            The later part of the proof is exactly similar to Theorem~\ref{Th3}.
    \end{proof}    
    \begin{corollary}
            For $n_2=1$, the signless Laplacian characteristics polynomial is the same as the Laplacian characteristics polynomial of the corona product of two graphs.
        \end{corollary}
        \begin{proof}
            The proof is exactly similar to Corollary~\ref{C1}.
        \end{proof}

        \begin{corollary}
         Let $\mathcal{G}_1$ and $\mathcal{G}_2$ be two co-spectral graphs and $\mathcal{G}$ be any arbitrary  graph, then 
            \begin{enumerate}
                \item $\mathcal{G}_1\circledast \mathcal{G}$ and $\mathcal{G}_2\circledast \mathcal{G}$ are $Q$-co-spectral.
                \item $\mathcal{G}\circledast \mathcal{G}_1$ and $\mathcal{G}\circledast \mathcal{G}_2$ are $Q$-co-spectral if $\chi_{Q_{\mathcal{G}_1}}(x)=\chi_{Q_{\mathcal{G}_2}}(x).$
            \end{enumerate}
     \end{corollary}
     
\section{Applications}
    \subsection{Kirchhoff Index}
    \noindent In 1993 Klein and Randić~\cite{klein1993rd} proposed a novel concept known as resistance distance. This is based on the electric resistance of a network that corresponds to a graph, where the resistance distance between any two adjacent vertices is 1 ohm. When summed over all pairs of vertices in a graph, this resistance distance serves as a new graph invariant. The Kirchhoff index, which is used to calculate electric resistance through Kirchhoff laws, is also defined for a graph $\mathcal{G}$ with $n(>2)$ vertices, as follows: $$Kf(\mathcal{G})=n\sum_{i=2}^n\frac{1}{\mu_i}.$$
    This renowned connection between the Laplacian spectrum and the Kirchhoff index was determined in 1996 by Zhu et al.~\cite{zhu1996extensions} and Gutman and Mohar~\cite{gutman1996quasi}.

    \begin{theorem}
        Let $\mathcal{G}_1$ and $\mathcal{G}_2$ be two regular graphs of order $n_1$ and $n_2$ with regularity $r_1$ and $r_2$ respectively. Also let the Laplacian spectrum be $0=\mu_{i1}\leq\mu_{i2}\leq\ldots\leq\mu_{in_i};~ i=1,2.$ Then $$Kf(\mathcal{G}_1\circledast\mathcal{G}_2)=2n_1n_2\Bigg[\frac{n_1(n_2-1)}{n_2(1+r_1)}+\sum_{i=1}^{n_1}\frac{2+\mu_{1i}}{n_2\mu_{1i}}+\sum_{i=2}^{n_2}\frac{n_1}{n_2+\mu_{2i}}\Bigg].$$
    \end{theorem}

    \begin{proof}
        From equation (\ref{Eqn4}), we have the Laplacian characteristics polynomial as
        $$f(L_{\mathcal{G}_1\circledast \mathcal{G}_2},x)=(x-r_1n_2-n_2)^{n_1(n_2-1)}f(L_{\mathcal{G}_2},(x-n_2))^{n_1}\cdot \prod_{i=1}^{n_1}\big[x-n_2-n_2\mu_{1i}-n_2\chi_{L_{\mathcal{G}_2}}(x-n_2).$$
        Now, each row sum of the Laplacian matrix of $\mathcal{G}_2$ being $0$, the Laplacian coronal of $\mathcal{G}_2$ is given by $$\chi_{L_{\mathcal{G}_2}}=\frac{n_2}{x}.$$
        So, \begin{equation}\label{Eqn5}
            \begin{split}
                f(L_{\mathcal{G}_1\circledast \mathcal{G}_2},x)&=(x-r_1n_2-n_2)^{n_1(n_2-1)}f(L_{\mathcal{G}_2},(x-n_2))^{n_1}\cdot \prod_{i=1}^{n_1}\big[x-n_2-n_2\mu_{1i}-n_2\frac{n_2}{x-n_2}\big]\\
                &=(x-r_1n_2-n_2)^{n_1(n_2-1)}f(L_{\mathcal{G}_2},(x-n_2))^{n_1}\cdot \prod_{i=1}^{n_1}\frac{1}{x-n_2}\big[x^2-n_2(2+\mu_{1i})x+n_2^2\mu_{1i}\big].
            \end{split}
        \end{equation}
        Therefore, the roots of the above Laplacian characteristic polynomial are:
        \begin{enumerate}
            \item $n_2(1+r_1)$ repeated $n_1(n_2-1)$ times.
            \item $n_2+\mu_{22},n_2+\mu_{23},\ldots,n_2+\mu_{2n_2}$ with multiplicity $n_1$. \Big($n_2+\mu_{21}$ is not a root as $n_2$ is a pole of $\chi_{L(\mathcal{G}_2)}(x-n_2)$\Big).
            \item Roots of the polynomial $x^2-n_2(2+\mu_{1i})x+n_2^2\mu_{1i},$ for $i=1,2,\ldots,n_1.$
        \end{enumerate}
        Let $\alpha$ and $\beta$ be the roots of the equation 
        \begin{equation}\label{Eqn6}
            x^2-n_2(2+\mu_{1i})x+n_2^2\mu_{1i}.
        \end{equation}
        Then,
        \begin{equation*}
            \begin{split}
                \frac{1}{\alpha}+\frac{1}{\beta}&=\frac{\alpha+\beta}{\alpha\beta}\\
                &=\frac{n_2(2+\mu_{1i})}{n_2^2\mu_{1i}}\\
                &=\frac{(2+\mu_{1i})}{n_2\mu_{1i}}.
            \end{split}
        \end{equation*}
        Hence, the Kirchhoff index of $\mathcal{G}_1\circledast\mathcal{G}_2$ is 
        $$Kf(\mathcal{G}_1\circledast\mathcal{G}_2)=2n_1n_2\Bigg[\frac{n_1(n_2-1)}{n_2(1+r_1)}+\sum_{i=1}^{n_1}\frac{2+\mu_{1i}}{n_2\mu_{1i}}+\sum_{i=2}^{n_2}\frac{n_1}{n_2+\mu_{2i}}\Bigg].$$
        
    \end{proof}

    \subsection{Spanning Tree}
        The concept of a spanning tree~\cite{cvetkovic1980spectra} refers to a tree that is also a subgraph of a given graph $\mathcal{G}$. The number of spanning trees for a graph $\mathcal{G}$ is represented by $t(\mathcal{G})$. If $\mathcal{G}$ is a connected graph with $n$ vertices and Laplacian eigenvalues $0=\mu_1\leq\mu_2\leq\ldots\leq\mu_n$, then the number of spanning trees can be calculated using the formula: $$t(\mathcal{G})=\frac{\mu_2(\mathcal{G})\mu_3(\mathcal{G})\ldots\mu_n(\mathcal{G})}{n}.$$

        \begin{theorem}
            Let $\mathcal{G}_1$ and $\mathcal{G}_2$ be two regular graphs of order $n_1$ and $n_2$ with regularity $r_1$ and $r_2$ respectively. Also let the Laplacian spectrum be $0=\mu_{i1}\leq\mu_{i2}\leq\ldots\leq\mu_{in_i};~ i=1,2.$ Then the number of spanning trees is given by, $$t(\mathcal{G}_1\circledast\mathcal{G}_2)=\frac{1}{2n_1n_2}\Big[\{n_2(1+r_1)\}^{n_1(n_2-1)}\cdot\prod_{i=2}^{n_2}(n_2+\mu_{2i})^{n_1}\cdot\prod_{i=1}^{n_1}n_2^2\mu_{1i}\Big].$$
        \end{theorem}

        \begin{proof}
            Let $\alpha,$ $\beta$ be two roots of equation (\ref{Eqn6}) i.e., $x^2-n_2(2+\mu_{1i})x+n_2^2\mu_{1i}.$
            Then, $\alpha\beta=n_2^2\mu_{1i}.$\\
            Applying the definition of the number of spanning trees in equation (\ref{Eqn5}), we obtain: $$t(\mathcal{G}_1\circledast\mathcal{G}_2)=\frac{1}{2n_1n_2}\Big[\{n_2(1+r_1)\}^{n_1(n_2-1)}\cdot\prod_{i=2}^{n_2}(n_2+\mu_{2i})^{n_1}\cdot\prod_{i=1}^{n_1}n_2^2\mu_{1i}\Big].$$
        \end{proof}

    \subsection{Wiener Index}
        The Wiener index~\cite{wiener1947structural}, written as $W(\mathcal{G})$, is the total distance between every pair of vertices in $\mathcal{G}$, i.e., $W(\mathcal{G})=\sum_{i<j}d(v_i,v_j).$ It can also be represented in terms of the distance matrix$(D(\mathcal{G}))$ as $W(\mathcal{G})=\frac{1}{2}\big[\textbf{1}_n^TD(\mathcal{G})\textbf{1}_n\big].$

        \begin{theorem}
            The Wiener index of the product graph $\mathcal{G}_1\circledast\mathcal{G}_2$ can be represented in terms of Wiener index of $\mathcal{G}_1$ and $\mathcal{G}_2$ as follows: $$W(\mathcal{G}_1\circledast\mathcal{G}_2)=4n_2^2W(\mathcal{G}_1)+n_1W(\mathcal{G}_2)+n_1n_2(2n_1n_2-1).$$
        \end{theorem}

        \begin{proof}
            Following the definition~\ref{Def1} of $\mathcal{G}_1\circledast \mathcal{G}_2$, we have
            $D(\mathcal{G}_1\circledast \mathcal{G}_2)=$ $$\begin{bmatrix} 
                (I_{n_2}\otimes D(\mathcal{G}_1))+(J_{n_2}-I_{n_2})\otimes (2I_{n_1}+D(\mathcal{G}_1)) & \textbf{1}_{n_2}^T\otimes(D(\mathcal{G}_1+J_{n_1}))\otimes \textbf{1}_{n_2}\\
                \textbf{1}_{n_2}\otimes(D(\mathcal{G}_1+J_{n_1}))\otimes \textbf{1}_{n_2}^T & I_{n_1}\otimes D(\mathcal{G}_2)+\{2(J_{n_1}-I_{n_1})+D(\mathcal{G}_1)\}\otimes J_{n_2}
            \end{bmatrix}.$$
            Then, using the formula mentioned above, we have,
            \begin{equation*}
                \begin{split}
                    W(\mathcal{G}_1\circledast\mathcal{G}_2)&=\frac{1}{2}\big[\textbf{1}_{2n_1n_2}^TD(\mathcal{G}_1\circledast\mathcal{G}_2)\textbf{1}_{2n_1n_2}\big]\\
                    &=\frac{1}{2}\big[n_2\cdot2W(\mathcal{G}_1)+n_2(n_2-1)\cdot(2n_1+2W(\mathcal{G}_1))+2n_2^2(2W(\mathcal{G}_1)+n_1^2)+n_1\cdot2W(\mathcal{G}_2)\\ &+\{2n_1(n_1-1)+2W(\mathcal{G}_1)\}n_2\big]\\
                    &=4n_2W(\mathcal{G}_1)+n_1W(\mathcal{G}_2)+n_1n_2(2n_1n_2-1).
                \end{split}
            \end{equation*}
        \end{proof}

    \subsection{Integral Graph}
    \noindent In this section, we discuss the condition for the graph product to be Laplacian and signless Laplacian Integral.

    \begin{theorem}
        Let $i=1,2$, $\mathcal{G}_i$ be two $r_i$-regular graph with $n_i$ vertices. Then $\mathcal{G}_1\circledast \mathcal{G}_2$ is Laplacian integral graph if and only if $\mathcal{G}_2$ is Laplacian Integral and for $i=1,2,\ldots,n_1$, the roots of the polynomial $\big[x-n_2-n_2\mu_{1i}-n_2\chi_{L_{\mathcal{G}_2}}(x-n_2)\big]$ are integers.
    \end{theorem}
    \begin{proof}
        The proof is obvious from Theorem~\ref{Th3}.
    \end{proof}

    \begin{theorem}
        Let $i=1,2$, $\mathcal{G}_i$ be two graph with $n_i$ vertices. Then $\mathcal{G}_1\circledast \mathcal{G}$ is signless Laplacian integral graph if and only if $\mathcal{G}_2$ is signless Laplacian Integral and for $i=1,2,\ldots,n_1$, the roots of the polynomial $\big[x-n_2-n_2\nu_{1i}-n_2\chi_{Q_{\mathcal{G}_2}}(x-n_2)\big]$ are integers.
    \end{theorem}
    \begin{proof}
        The proof is obvious from Theorem~\ref{Th4}.
    \end{proof}

\section*{Acknowledgement}
        We want to acknowledge the National Institute of Technology Sikkim for awarding Bishal Sonar a doctoral fellowship.

    \bibliographystyle{abbrv}
    \bibliography{main.bib}

\end{document}